\documentclass{article}

\usepackage{enumerate}
\usepackage{amsmath,amssymb,amsthm}
\usepackage{tikz}

\newtheorem{thm}[equation]{Theorem}
\newtheorem*{thm*}{Theorem}
\newtheorem{lem}[equation]{Lemma}

\theoremstyle{definition}

\newtheorem*{rmk}{Remark}

\numberwithin{equation}{section}

\DeclareMathOperator{\comass}{comass}
\DeclareMathOperator{\mass}{mass}
\DeclareMathOperator{\vol}{vol}
\DeclareMathOperator{\id}{id}
\DeclareMathOperator{\Lip}{Lip}
\DeclareMathOperator{\poly}{poly}
\newcommand{\ZZ}{\mathbb{Z}}
\newcommand{\RR}{\mathbb{R}}

\newcommand{\epsi}{\varepsilon}

\begin{document}
\title{A hardness of approximation result in metric geometry}
\author{Zarathustra Brady, Larry Guth, and Fedor Manin\thanks{MSC 2010 classes: 53C23, 68Q17.  Keywords: hyperspherical radius, hardness of approximation.}}
\maketitle
\begin{abstract}
  We show that it is $\mathsf{NP}$-hard to approximate the hyperspherical radius
  of a triangulated manifold up to an almost-polynomial factor.
\end{abstract}

\section{Introduction}

This paper is concerned with the problem of estimating the smallest Lipschitz
constant of a map with a given degree.  We consider this problem from the point
of view of computational complexity.

Suppose that $\Sigma$ is a closed oriented triangulated manifold of dimension
$n$.  We equip $\Sigma$ with a metric by making each simplex a unit equilateral
Euclidean simplex.   We let $L_{\neq 0}(\Sigma)$ and $L_1(\Sigma)$ denote the
smallest Lipschitz constant of a map from $\Sigma$ to the unit $n$-sphere with
non-zero degree and with degree 1, repsectively.  When $\Sigma$ is a triangulated
2-sphere, the paper \cite{GuthSurf} gives a polynomial time algorithm to estimate
$L_{\neq 0}(\Sigma)$ up to a constant factor.  In contrast, if $\Sigma$ is a
triangulated surface of arbitrary degree, we show that it is $\mathsf{NP}$-hard
to approximate $L_{\neq 0}(\Sigma)$ up to a constant factor, or indeed to within a
far wider range.  Similarly, if $n \ge 3$ and $\Sigma$ is a triangulation of
$S^n$, then we show that it is $\mathsf{NP}$-hard to estimate
$L_{\neq 0}(\Sigma)$.  Here is a more precise and general statement.  

\begin{thm} \label{thm:main}
  Let $n \geq 2$.  For $\Sigma$ a triangulation of $S^n$ (if $n \geq 3$) or a
  triangulated surface (if $n=2$), both $L_{\neq 0}(\Sigma)$ and $L_1(\Sigma)$ are
  $\mathsf{NP}$-hard to approximate to within $N^{c/\log\log N}$, where
  $N=\vol \Sigma$ and $c>0$ depends on $n$.
\end{thm}

Compare the total range of possible values for these quantities.  The
``roundest'' possible shape for $\Sigma$ is approximately a round $n$-sphere of
radius $N^{1/n}$, and in the worst case one can contract all but a single
simplex.  Therefore,
$$N^{-1/n} \lesssim L_{\neq 0}(\Sigma) \lesssim L_1(\Sigma) \lesssim 1.$$

\paragraph*{Background.}
In the 1970s Gromov began to investigate the smallest Lipschitz constant of a map
in a given homotopy class \cite{GrHED}.  For instance, he showed that the
smallest Lipschitz constant of a degree $D$ map from the unit $n$-sphere to
itself is $\sim D^{1/n}$.  In \cite{GrLa}, Gromov and Lawson defined the
hyperspherical radius of a (closed oriented) Riemannian $n$-manifold
$(\Sigma^n, g)$ as the largest radius $R$ so that $(\Sigma^n, g)$ admits a
1-Lipschitz map to the round $n$-sphere of radius $R$ with non-zero degree.  The
hyperspherical radius of $\Sigma$ is $1/L_{\neq 0}(\Sigma)$.  Gromov and Lawson
proved that if the scalar curvature of $g$ is at least 1, and $\Sigma$ is spin,
then the hyperspherical radius of $(\Sigma^n, g)$ is at most a constant $C(n)$.  

The most well-known open problem about hyperspherical radius involves the
universal covers of aspherical manifolds.  Suppose that $\Sigma$ is a closed
manifold.  If we choose a metric $g$ on $\Sigma$, then we can consider the
hyperspherical radius of the universal cover $(\tilde \Sigma, \tilde g)$.
Different metrics $g$ may lead to different hyperspherical radii, but whether the
hyperspherical radius is infinite does not depend on the metric.  For instance,
on $S^1 \times S^2$, the hyperspherical radius of the universal cover is always
finite, but on $T^3$ the hyperspherical radius of the universal cover is always
infinite.  It is an open question whether every closed aspherical manifold has
universal cover with infinite hyperspherical radius.  This open question is
connected to scalar curvature (cf.\ \cite{GrLa}) and to the Novikov conjecture
(cf.\ \cite{DrFeWe}).

This problem is difficult partly because the hyperspherical radius is harder to
estimate than it may sound at first.  Our main theorem shows from a different
perspective that hyperspherical radius is hard to estimate.  That said, our
result is the best possible for example in the sense that in exponential time,
hyperspherical radius can be estimated to within a constant; we make this precise
in an appendix.  This contrasts, for example, with the related problem of
determining the least Lipschitz constant of a degree nonzero map between two
simplicial manifolds, both given as input.  In that case, the existence of such a
map is undecidable, and even if it exists, its minimal Lipschitz constant may
grow faster than any computable function of the volumes of the two manifolds.
These observations are rooted in the work of Nabutovsky and Weinberger,
cf.\ \cite{CRM}.


\paragraph*{Idea of proof.}
The proof of Theorem \ref{thm:main} builds on recent progress in quantitative
topology due to Chambers, Dotterrer, Ferry, Weinberger, and the last
author---cf.\ \cite{CDMW}, \cite{CMW}, and \cite{PCDF}.  This work answered a
number of questions from quantitative topology that were raised by Gromov in the
1990s.  For instance, they prove that if $f$ is a contractible map from the unit
$m$-sphere to the unit $n$-sphere with Lipschitz constant $L$, then $f$ extends
to a nullhomotopy $F: B^{m+1} \rightarrow S^n$ with Lipschitz constant at most
$C(m,n)(L+1)$ (see \cite{CDMW} for the case when $n$ is odd and \cite{Berd} and
\cite{Resc} for the case when $n$ is even).  This progress involves new ideas
about how to construct maps with controlled Lipschitz constant and we build on
those ideas in our proof.

Let us briefly describe the examples of high genus surfaces $\Sigma$ where
$L_{\neq 0}(\Sigma)$ is hard to estimate.  We begin with a complicated simplicial
complex $X$---in particular the dimension of $H_2(X; \mathbb{Z})$ is quite large.
The surface $\Sigma$ sits inside of $X$.  Topologically, it represents a nonzero
homology class $[\Sigma] \in H_2(X, \ZZ)$.  As a metric space, $\Sigma$ is very
similar to $X$: it comes within a distance $\epsilon$ of every point of $X$, and
moreover, for any two points in $\Sigma$, the distance between them in $\Sigma$
is almost the same as the distance between them in $X$ (using the path metric on
both spaces).  This last feature requires $\Sigma$ to have high genus---we can
achieve it by adding lots of handles to provide short routes between various
points.

Since $\Sigma$ and $X$ are so similar as metric spaces, if $f$ is a map from
$\Sigma$ to the unit $2$-sphere with Lipschitz constant not too large, then $f$
extends to a map $g: X \rightarrow S^2$ with a similar Lipschitz constant.  Each
map $g: X \rightarrow S^2$ induces a cohomology class
$\alpha = g^*([S^2]^*) \in H^2(X, \ZZ)$.  For each choice of
$\alpha \in H^2(X, \ZZ)$, let $\Lip(\alpha)$ be the smallest Lipschitz constant
of a map $g: X \rightarrow S^2$ with $g^*([S^2]^*) = \alpha$.  Using the recent
ideas in quantitative topology we described above, we are able to accurately
estimate $\Lip(\alpha)$ for each $\alpha \in H^2(X, \ZZ)$.  But
$L_{\neq 0}(\Sigma)$ is approximately the minimum of $\Lip(\alpha)$ over all
$\alpha \in H^2(X, \ZZ)$ with $\alpha([\Sigma]) \neq 0$.  This is a minimization
problem over $H^2(X, \ZZ)$, which is isomorphic to $\ZZ^D$ for a large $D$.  It
turns out to be closely related to the shortest vector problem for lattices in
$L^\infty$.  Dinur \cite{Dinur} has a hardness of approximation result for the
shortest vector problem in $L^\infty$, and using her paper we get our hardness of
approximation result for $L_{\neq 0}$.  

\paragraph*{Further questions.}
There are many quantities studied in metric geometry which it is not obvious how
to estimate, and it would be interesting to understand the computational
complexity of estimating them to various degrees of accuracy.  For instance, it
unknown how hard it is to approximate Uryson widths (cf. \cite{GrWRI}) or minimax
volumes (cf. \cite{GuMV}).   As another example, given a simplicial complex $X$
and a homology class $h \in H_k(X, \ZZ)$, the mass of $h$ is defined to be the
smallest mass of an integral $k$-cycle in the class $h$ (and the mass of a
$k$-cycle is its $k$-dimensional volume counted with multiplicity).  How hard is
it to approximate the mass of a homology class? 

\paragraph*{Organization.}
The paper is organized as follows.  In Section 2 we introduce the comass and
explain how it is related to Lipschitz constants of maps.  In Section 3, we
efficiently approximate the comass using linear programming.  In Section 4, we
use these tools to prove a hardness of approximation result for homologically
non-trivial maps from a complicated complex to the unit sphere.  In Section 5, we
prove our main theorem for the case $n=2$.  In Section 6, we prove our main
theorem for triangulations of $S^n$ with $n \ge 3$.  Finally, we have also
included an appendix explaining some upper bounds on the complexity of computing
hyperspherical radius.

\paragraph*{Acknowledgements.}
We would like to thank an anonymous referee for correcting a number of typos and
offering other useful suggestions.  F.\ Manin was partially supported by NSF
individual grant DMS-2001042.

\section{Comass and the Lipschitz constant}

Let $X$ be a finite $m$-dimensional simplicial complex endowed with the standard
simplexwise linear metric with edge length 1.  In this section, we discuss the
geometric relationship between maps $X \to S^n$ and the cohomology of $X$.  The
discussion applies when $n$ is odd, or when $n$ is even and $m \leq 2n-2$.

For each homotopy class $\alpha \in [X, S^n]$, we associate $\alpha^*[S^n]$.
This association gives a map from $[X, S^n]$ to $H^n(X; \ZZ)$.   Obstruction
theory shows that this map is uniformly finite-to-one: the cardinality of a point
preimage is bounded by
$$\prod_{k=n+1}^{\dim X} \#\{i\text{-simplices of }X\} \cdot \#\pi_k(S^n).$$
Also, the image of this map contains a finite-index subgroup of
$H^n(X;\mathbb{Z})$, since there is a map $K(\mathbb{Z},n)^{(m)} \to S^n$ of
positive degree. 

In this section, we study the relationship between the geometric properties of a
homotopy class $\alpha \in [X, S^n]$ and the geometric properties of the
cohomology class $\alpha^*[S^n]$.   We will study the Lipschitz norm
$\Lip(\alpha)$ of a class $\alpha \in [X, S^n]$ and compare it to the comass norm
of the cohomology class $\alpha^*[S^n]$.

First we recall the definition of the comass norm.  The comass norm is a norm on
$H^n(X; \RR)$.  Recall that a (real) Lipschitz $n$-chain is a formal sum
$\sum_i a_i \phi_i$ where $a_i \in \mathbb{R}$ and $\phi_i:\Delta^n \rightarrow X$
is a Lipschitz map.  Lipschitz chains can be used to define homology just as well
as singular (i.e.\ continuous) chains.  One advantage of Lipschitz chains is that
there is a volume associated to each $\phi_i$.  The mass of a chain is then given
by
$$\mass\left(\sum_i a_i\phi_i\right)=\sum_i |a_i| \vol(\phi_i).$$
The mass descends to a norm on $H_n(X,\mathbb{R})$, written $\|\cdot\|_{\mass}$.
The dual norm on $H^n(X;\mathbb{R})$ is called the comass---in other words,
$$\|w\|_{\comass} := \sup_{\substack{h \in H_n(X,\mathbb{R})\\\|h\|_{\mass}=1}} |w(h)|.$$
\begin{thm} \label{lip=comass}
  Let $n \geq 2$ and let $X$ be a simplicial complex.  If $n$ is odd or
  $\dim X \leq 2n-2$, then for every homotopy class of maps
  $\alpha \in [X,S^n]$
  $$(\vol(S^n)\|\alpha^*[S^n]\|_{\comass})^{1/n} \leq \Lip(\alpha) \leq
  C(\dim X,n)(\|\alpha^*[S^n]\|_{\comass}^{1/n}+1).$$
  When $n$ is even and $\dim X=2n-1$, the first inequality holds, but the second
  holds only in a weaker sense: for every element $a \in H^n(X;\mathbb{R})$ which
  is in the image of $[X,S^n]$, there is \emph{some} $\alpha \in [X,S^n]$ such
  that $\alpha^*[S^n]=a$ and
  $$\Lip(\alpha) \leq C(\dim X,n)(\|a\|_{\comass}^{1/n}+1).$$
\end{thm}
Notice that if the comass is greater than $1$, this theorem allows us to compute
the Lipschitz norm up to a multiplicative constant; on the other hand, if the
comass is tiny, it gives almost no information.  If $\Sigma$ is a triangulated
$n$-manifold, and $\alpha \in [\Sigma, S^n]$ is the class of degree 1 maps, then 
the bounds given by this theorem are not very strong.  In this case,
$\|\alpha^*[S^n]\|_{comass}=\|[\Sigma]\|_{comass} = 1/\vol(\Sigma)$.  Theorem
\ref{lip=comass} then gives $\vol(\Sigma)^{-1/n}\lesssim \Lip(\alpha) \lesssim 1$.
In Section \ref{S:hard} we will see some homotopy classes where Theorem
\ref{lip=comass} gives sharp bounds.

Theorem \ref{lip=comass} is a corollary of the more general shadowing principle
in quantitative homotopy theory, \cite[Thm.~4.1]{PCDF}, given by the last
author.  However, it is simpler than the general shadowing principle, and all the
tools are already present in the second author's survey paper \cite{RPQT}.  We
give a full proof here which may also be of use as an introduction to
\cite{PCDF}.

\begin{proof}
  Let $f:X \to S^n$ be any representative of the homotopy class $\alpha$.

  We first show the left inequality.  For any $h \in H_n(X;\mathbb{R})$, we have
  $$|f_*h| \leq \frac{(\Lip f)^n\|h\|_{\mass}}{\vol(S^n)};$$
  By definition, this gives
  $$(\vol(S^n)\|\alpha^*[S^n]\|_{\comass})^{1/n} \leq \Lip(\alpha).$$
  Note that this does not depend on the dimension of $X$.

  To show the second inequality when $n$ is odd or $\dim X \leq 2n-2$, we will
  homotope $f$ to a map with the appropriate Lipschitz bound.  Define
  $$R=\lceil\|\alpha^*[S^n]\|_{\comass}^{1/n}\rceil,$$
  and let $X_R$ be a subdivision of $X$ at scale $\sim 1/R$: specifically, we
  ensure that the volume of each $n$-simplex of $X_R$ is
  $\leq \|\alpha^*[S^n]\|_{\comass}^{-1}$, and that all simplices are
  $C(\dim X,n)$-bilipschitz to the standard linear simplex with edge lengths
  $1/R$.  There are many ways to do this \cite[\S2]{CDMW}.

  If the comass is very small, the best we can do is no subdivision at all, hence
  the theorem does not tell us very much in that instance.

  We will homotope $f:X_R \to S^n$ to a map which, for every $k$, restricts to on
  each $k$-simplex of $X_R$ to one of a finite set $\mathcal{F}_k$ of maps.  This
  automatically gives the Lipschitz bound we are looking for.  Specifically:
  \begin{itemize}
  \item For $k \leq n-1$, $\mathcal{F}_k$ consists of a single map, the constant
    map to the basepoint of $S^n$.
  \item For $k=n$, $\mathcal{F}_n$ contains a map
    $(\Delta^n,\partial\Delta^n) \to S^n$ of each degree between $-(n+2)$ and
    $(n+2)$.
  \item For $k>n$, $\mathcal{F}_k$ contains a map for each map
    $\partial\Delta^k \to S^n$ which restricts to an element of
    $\mathcal{F}_{k-1}$ on each $(k-1)$-simplex and relative homotopy class of
    extensions to $\Delta^k$.  The dimension restriction means that this is a
    finite set.
  \end{itemize}
  We will build a sequence of maps $f_k$ homotopic to $f$ such that the
  restrictions of $f_k$ to $k$-simplices of $X_R$ are chosen from
  $\mathcal{F}_k$.

  First we homotope $f$ to a map $f_{n-1}:X_R \to S^n$ which takes $X_R^{(n-1)}$ to
  the basepoint of $S^n$.

  We now homotope this $f_{n-1}$ to a map $f_n$ which has degree $0$ or $\pm 1$ on
  every $n$-simplex of $X_R$.  This is done as follows.  Note that $f_{n-1}$ has a
  well-defined degree on each $n$-simplex of $X_R$.  This defines a simplicial
  cocycle $z \in C^n(X_R;\mathbb{Z})$.

  Now notice that the comass of $\alpha^*[S^n]$ is at least as large as its
  maximum value on \emph{simplicial} $n$-cycles of mass $\leq 1$ in
  $Z_n(X_R;\mathbb{R})$.  In particular, this includes all simplicial cycles
  whose $\ell^1$-norm with respect to the obvious basis is
  $\leq \|\alpha^*[S^n]\|_{\comass}$.  By the Hahn--Banach theorem, this means that
  the $\ell^\infty$-minimal simplicial representative
  $\hat z \in C^n(X_R;\mathbb{R})$ of $\alpha^*[S^n]$ satisfies
  $\lVert \hat z \rVert_\infty \leq 1$.

  Now we can find a chain $b \in C^{n-1}(X;\mathbb{R})$ such that $db=z-\hat z$.
  Let $\hat b=[b] \in C^{n-1}(X;\mathbb{Z})$ be the nearest integer chain to $b$.
  Then the chain $z-d\hat b$ satisfies
  $$\lVert z-d\hat b \rVert_\infty \leq n+2.$$

  We now define a homotopy $h_n:X_R \times [0,1] \to S^n$ between $f_{n-1}$ and a
  new map $f_n$ on cells of the product cell structure on $X_R \times [0,1]$,
  with $[0,1]$ thought of as the 1-simplex:
  \begin{itemize}
  \item Send $(X_R \times [0,1])^{(n-1)}$ to the basepoint of $S^n$.
  \item Map $n$-cells of the form $q \times [0,1]$, where $q$ is an
    $(n-1)$-simplex of $X_R$, to $S^n$ via a map of degree
    $\langle \hat b, q \rangle$
  \item This forces the degree on $p \times \{1\}$, where $p$ is an $n$-simplex
    of $X_R$, to be
    $$\langle z-d\hat b, p\rangle.$$
    We let $f_n|_p$ be the corresponding element of $\mathcal{F}_n$.
  \item Finally, we extend arbitrarily to the rest of $X_R \times [0,1]$.
  \end{itemize}

  Now, suppose we have constructed $f_k:X_R \to S^n$, $k \geq n$.  We homotope to
  $f_{k+1}$ by homotoping, on each $(k+1)$-simplex $p$, to the element of
  $\mathcal{F}_{k+1}$ corresponding to the relative homotopy class of $f_k|_p$.
  Then we extend the homotopy arbitrarily to higher skeleta.

  The map $f_{\dim X}$ is the map we want.

  Finally, suppose that $\dim X=2n-1$ and let $f:X \to S^n$ be some map.  Letting
  $R$ be as before, we can proceed with the same construction of $f_{2n-2}$ on the
  $(2n-2)$-skeleton of $X_R$.  Finally, in the last step there is no obstruction
  to filling in the map on the $(2n-1)$-cells.  Therefore, for each map
  $\partial\Delta^k \to S^n$ which restricts to an element of $\mathcal{F}_{2n-2}$
  on each $(2n-2)$-simplex, we can fix a filling, and then use those fillings to
  extend $f_{2n-2}$ to the $(2n-1)$-cells of $X_R$ with bounded Lipschitz
  constant.  The resulting map is not homotopic to $f$, but induces the same
  class in $H^n(X;\mathbb{R})$.
\end{proof}

\section{Computing the comass} \label{S:comass}

A good feature of the comass is that it can be approximated efficiently using
linear programming.

\begin{lem}
  There is a polynomial-time algorithm which, given an $m$-dimensional simplicial
  complex $X$, an integer $n$, and an
  element $\beta \in H^n(X;\mathbb{R})$, computes a constant
  $\comass_\Delta(\beta)$ such that for some constant $C(m,n)$,
  $$\comass_\Delta(\beta) \leq \|\beta\|_{\comass}
  \leq C(m,n)\comass_\Delta(\beta).$$
\end{lem}
\begin{proof}
  The quantity $\comass_\Delta(\beta)$ is the \emph{simplicial comass}, that is,
  the supremum of $\langle\beta,z\rangle$ over simplicial cycles $z$ of mass 1.

  Clearly, the simplicial comass is less than or equal to the usual comass.
  Conversely, the Federer--Fleming inequality allows us to push any cycle to a
  simplicial one while only increasing the mass by a constant factor $C(m,n)$.
  This gives the second inequality.

  It remains to show that the simplicial comass is polynomially computable.  In
  fact, it is a linear programming problem: given a simplicial representative
  $b \in C^n(X;\mathbb{R})$ of $\beta$,
  $$\text{maximize }a(z)\quad\text{subject to}\quad
  \mass(z) \leq 1,\quad \partial z=0.$$
  We now give a bit more detail as to the implementation.  We introduce a
  variable $z_p \geq 0$ for each \emph{oriented} $n$-simplex $p$ of $X$ (that is,
  two for each unoriented simplex).  The $p$-coefficient of the chain $z$ is then
  given by $z_p-z_{p_-}$, where $p_-$ is the same simplex with opposite
  orientation.

  Then the constraint $\mass(z) \leq 1$ can be written as
  $$\vol(\Delta^n)\sum_p z_p \leq 1,$$
  since every chain has a representative where for every $n$-simplex $p$ either
  $z_p$ or $z_{p_-}$ is zero.  The constraint $\partial z=0$ can be written with
  an equation for each $(n-1)$-simplex of $X$.

  The total number of variables and constraints is linear in $|X|$; hence the
  computation is polynomial in $|X|$.
\end{proof}

Suppose that $\alpha \in [X, S^n]$.  In polynomial time, we can approximate the
comass of $\alpha^*[S^n]$ up to a constant factor.  Suppose that $n$ is
odd or $\dim X \le 2n-2$, so that we can apply Theorem \ref{lip=comass}.  If in
addition $\|\alpha^*[S^n]\|_{\text{comass}} \ge 1$, then Theorem \ref{lip=comass}
guarantees that $\|\alpha^*[S^n]\|_{\text{comass}}^{1/n}$ agrees with
$\Lip(\alpha)$ up to a constant factor $C(\dim X, n)$.  This gives a
significant class of examples in which we can approximate $\Lip(\alpha)$ up
to a constant factor in polynomial time.

\section{Hardness of approximation} \label{S:hard}

In the last two sections, we have given some conditions when it is easy to
estimate $\Lip(\alpha)$ for a given $\alpha \in [X, S^n]$.  We will use that work
as a tool to eventually give an example where it is hard to compute
$\Lip(\alpha)$.  In this section, we prove our first hardness of approximation
result, for a slightly different quantity.

For a simplicial complex $X$, we define $L_{\text{HNT}}(X)$ to be the smallest
Lipschitz constant of a map $f$ from $X$ to the unit $n$-sphere so that
$f^*[S^n]$ is non-zero in $H^n(X; \RR)$.  (HNT stands for ``homologically
non-trivial''.)

\begin{thm} \label{hardLHNT}
  Let $n \ge 2$.  Let $X$ denote a simplicial complex of dimension $(n+1)$.  It
  is $\mathsf{NP}$-hard to approximate $L_{\text{HNT}}(X)$ to within a factor of
  $N^{C/\log\log N}$, where $N$ is the number of simplices in $X$.  
\end{thm}

Let us say something about the proof before we begin the details. 
In the family of examples $X$ we construct,
\begin{equation} \label{big-comass} \tag{$*$}
  L_{\text{HNT}}(X) \sim \min\bigl\{\|\beta\|_{\text{comass}}^{1/n} \mid
  \text{all non-torsion }\beta \in H^n(X; \ZZ)\bigr\}.
\end{equation}
Estimating this minimum is a cousin of the shortest vector problem.  Dinur
\cite{Dinur} proved that a version of the shortest vector problem in $L^\infty$ is
$\mathsf{NP}$-hard to approximate.  Using her theorem, we will show that
$L_{\text{HNT}}$ is also $\mathsf{NP}$-hard to approximate.

Actually to get \eqref{big-comass}, we only need that every non-torsion
$\beta \in H^n(X;\ZZ)$ has comass at least some constant (in our case $1/(n+1)$.)
Then Theorem \ref{lip=comass} tells us that for individual $\alpha \in [X,S^n]$,
$\Lip(\alpha) \sim \|\alpha^*[S^n]\|_{\text{comass}}$; moreover, since $X$ is
$(n+1)$-dimensional, every $\beta \in H^n(X; \ZZ)$ is represented by some
$\alpha \in [X, S^n]$.

Eventually, we will prove that it is hard to approximate $L_{\neq 0}(\Sigma)$ and
$L_1(\Sigma)$, where $\Sigma$ is a triangulated manifold.  As a bridge to get to
that result, we will also prove that two cousins of $L_{\text{HNT}}(X)$ are hard to
approximate.  Suppose that $h \in H_n(X, \ZZ)$.  We define $L_{\neq 0}(X,h)$ to be
the smallest Lipschitz constant of a map $f: X \rightarrow S^n$ so that
$f^*[S^n](h) \neq 0$.  Similarly, we define we define $L_1(X,h)$ to be the
smallest Lipschitz constant of a map $f:X \rightarrow S^n$ so that
$f^*[S^n](h)=1$.

\begin{thm} \label{hardLXh}
  Let $n \ge 2$.  Let $X$ denote a simplicial complex of dimension $(n+1)$ and
  $h$ denote a homology class in $H_n(X, \ZZ)$.  It is $\mathsf{NP}$-hard to
  approximate $L_{\neq 0}(X,h)$ or $L_1(X,h)$ to within a factor of
  $N^{C/\log\log N}$, where $N$ is the number of simplices in $X$.  In fact, it is
  $\mathsf{NP}$-hard to distinguish the case where $L_{\neq 0}(X, h) \sim 1$ from
  the case where $L_{\neq 0}(X,h) \ge N^{C/\log \log N}$, and similarly for $L_1$.
\end{thm}

Now we begin the proofs.  All of the results are based on a hardness of
approximation result for shortest vectors in lattices in $L^\infty$.

\begin{thm} \label{lattices} (Dinur, \cite{Dinur})
  For a subgroup $\Gamma \subset \mathbb{Z}^N$ specified by generators of size
  polynomial in $N$, it is $\mathsf{NP}$-hard to approximate the smallest length
  of a non-zero vector in $\Gamma$ to within $O(N^{C/\log \log N})$.
\end{thm}

This would be enough to prove Theorem \ref{hardLHNT}, but we need a small
refinement to prove Theorem \ref{hardLXh}.  This refinement was also proven in
\cite{Dinur}: it follows from inspecting the proof of Theorem \ref{lattices}, and
in particular condition (a) appears in the proof of Proposition 22 of the paper.

\begin{thm} \label{latticesref}
  Consider a lattice $\Gamma \subset \ZZ^N$ with basis
  $\mathbf{u}_0, \mathbf{u}_1, \ldots, \mathbf{u}_M$, with
  $\|\mathbf{u}_j\|_\infty$ at most polynomial in $N$.  It is $\mathsf{NP}$-hard
  to distinguish the following two cases:
  \begin{enumerate}[(a)]	
  \item There is a vector $\mathbf{v} \in \Gamma$ with $\|\mathbf{v}\|_\infty=1$
    of the form $\mathbf{v}=\mathbf{u}_0+\sum_{j=1}^M a_j\mathbf{u}_j$, where
    $a_j \in \{0, 1\}$.
  \item There is no nonzero vector $\mathbf{v}$ in the subgroup with
    $\|\mathbf{v}\|_\infty \leq N^{C/\log \log N}$.
  \end{enumerate}
\end{thm}

Now we begin the proof of Theorems \ref{hardLHNT} and \ref{hardLXh}.  We start
with a lattice $\Gamma \in \ZZ^N$ with basis
$\mathbf{u}_0,\mathbf{u}_1,\ldots,\mathbf{u}_M$ as in Theorem
\ref{latticesref}.  We use this lattice to build an (n+1)-dimensional complex $X$
with an isomorphism $\gamma: \Gamma \rightarrow H^n(X; \ZZ)$, and so that the
$l^\infty$ norm on $\Gamma$ is closely related to the simplicial comass norm on
$H^n(X; \ZZ)$.  More precisely, we will show that
\begin{enumerate} [(a)]
\item If $\mathbf v \in \Gamma$ has the form $\sum_{j=0}^M a_j \mathbf u_j$ with
  $a_j \in \{0,1\}$, then
  $$\|\mathbf v\|_\infty=\comass_\Delta(\gamma(\mathbf v)).$$
\item For any $\mathbf v \in \Gamma$,
  $\comass_\Delta(\gamma(\mathbf v)) \geq \|\mathbf v\|_\infty$.
\end{enumerate}

Here is the construction of the complex $X$.  We start with a wedge of
$n$-spheres $\Sigma_0,\Sigma_1,\ldots,\Sigma_M$, each triangulated as
$\partial\Delta^{n+1}$.  For each coordinate of $\mathbb{Z}^N$, $i=1,\ldots, N$,
we attach a $n$-sphere $S_i$, also triangulated as $\partial\Delta^{n+1}$,
together with a triangulated mapping cylinder between $S_i$ and a map
$S^n \to \bigvee_{j=0}^M \Sigma_j$ in the homotopy class
$$\sum_{j=0}^M (n+1)u_{ji}\id_{\Sigma_j},\qquad
\text{where }\mathbf u_j=(u_{j0},\ldots,u_{jN}) \in \ZZ^N.$$

The complex $X$ deformation retracts to $\Sigma_0 \vee \cdots \vee \Sigma_M$, so
in particular $H^n(X;\ZZ)$ is the free abelian group generated by $[\Sigma_j]^*$.
We define $\gamma$ so that $\gamma(\mathbf u_j)=[\Sigma_j]^*$.  

Next we have to estimate the comass norm.  Let $\beta \in H^n(X;\mathbb{Z})$
satisfy $\beta(\Sigma_j)=\beta_j$; in other words,
$\beta=\gamma(\sum_{j=0}^M \beta_j\mathbf u_j)$.  Then
$$\beta([S_i]/(n+1))=\sum_{j=1}^M \beta_ju_{ji}$$
and therefore
$$\comass_\Delta(\beta) \geq
\sup_{1 \leq i \leq N} \biggl\lvert \sum_{j=1}^M \beta_ju_{ji} \biggr\rvert
=\biggl\lVert \sum_{j=1}^M \beta_j\mathbf u_j \biggr\rVert_\infty.$$
Conversely, if $z \in C_n(X;\mathbb{R})$ is a simplicial cycle of mass $1$, then
it is homologous to a combination of $S_i$'s and $\Sigma_j$'s with coefficients
summing up to at most $\frac{1}{n+1}$.  Then
$$\comass_\Delta(\beta)=\max\biggl\{\sup_j \frac{\beta_j}{n+1},
\sup_i \biggl\lvert\sum_{j=1}^M \beta_ju_{ji}\biggr\rvert\biggr\},$$
and if the $\beta_j$ are $0$ or $1$, the term on the right is larger.

Using Theorem \ref{latticesref} and our estimate connecting the comass of
$\gamma(\mathbf v)$ with the $l^\infty$ norm of $\mathbf v$ in $\ZZ^n$, we see
that it is $\mathsf{NP}$-hard to distinguish the following two cases:
\begin{enumerate} [(a)]	
\item There is a non-zero class $\beta \in H^n(X; \ZZ)$ with
  $\comass_\Delta(\beta) = 1$ and $\beta$ of the form
  $[\Sigma_0]^* + \sum_{j=1}^M a_j[\Sigma_j]^*$ with $a_j \in \{0, 1\}$.
\item For every non-zero $\beta \in H^n (X, \ZZ)$,
  $\comass_\Delta(\beta) \ge N^{C/\log \log N}$.	
\end{enumerate}

Since $X$ is homotopy equivalent to a wedge of $n$-spheres,
$[X, S^n] = H^n(X; \ZZ)$.  For each $\alpha \in [X, S^n]$, there is a unique
$\mathbf v_\alpha \in \Gamma$ so that $\gamma(\mathbf v_\alpha)=\alpha^*[S^n]$.  We
apply Theorem \ref{lip=comass} to estimate $\Lip(\alpha)$.  If
$\alpha \neq 0$, then we know that
$$\|\alpha^*[S^n]\|_{\text{comass}} \ge \|\mathbf v_\alpha\|_\infty \ge 1.$$
Theorem \ref{lip=comass} tells us that
$\Lip(\alpha) \sim \|\alpha^*[S^n]\|_{\text{comass}}^{1/n}$.  Therefore, we see that
$$L_{\text{HNT}}(X)^n \sim \min_{0 \neq \beta \in H^n(X; \ZZ)} \|\beta\|_{\text{comass}}.$$
But it is $\mathsf{NP}$-hard to distinguish the case when the right-hand side is
$\sim 1$ from the case when the right-hand side is $\ge N^{C/\log \log N}$.  This
shows Theorem \ref{hardLHNT}.

To prove Theorem \ref{hardLXh}, let $X$ be the same complex and let
$h=[\Sigma_0]$.  Then
$$L_{\neq 0}(X,h)^n \sim
\min_{\beta=\sum_{j=0}^M a_j[\Sigma_j]^*, a_0 \neq 0}\|\beta\|_{\text{comass}},$$
and
$$L_1(X,h)^n \sim
\min_{\beta=\sum_{j=0}^M a_j[\Sigma_j]^*, a_0=1} \|\beta\|_{\text{comass}}.$$

In both cases, we have seen that it is $\mathsf{NP}$-hard to distinguish the case
when the right-hand side is $\sim 1$ from the case when the right-hand side is
$\ge N^{C/\log \log N}$.  This shows Theorem \ref{hardLXh}.

\subsection{On the proof of Dinur's result}

For the sake of completeness, we give an idea of Dinur's proof of Theorem
\ref{lattices}.  This is a reduction which is performed in two steps.  The first
step reduces $\mathrm{SAT}$ to an intermediate optimization problem
$\mathcal{SSAT}_\infty$ (which is somewhat complicated to define) and the second
reduces $\mathcal{SSAT}_\infty$ to $\mathrm{SVP}_\infty$, the problem of deciding
the length of the $\ell^\infty$-shortest nonzero vector in a lattice.  The first
reduction creates a gap, that is, any satisfiable instance of $\mathrm{SAT}$
produces an instance of $\mathcal{SSAT}_\infty$ whose smallest solution has norm
1, and any unsatisfiable instance produces an instance of $\mathcal{SSAT}_\infty$
whose smallest solution has norm $\geq g \sim N^{c/\log \log N}$; and the second
reduction maintains this gap.

The second reduction, whose particulars were used to state Theorem
\ref{latticesref}, is essentially an act of tidying up: as we'll see, the
solution set to an instance of $\mathcal{SSAT}_\infty$ already looks like a
lattice, and the reduction adds an extra degree of freedom and massages the
vectors so that the norm in $\mathcal{SSAT}_\infty$ (a complicated combination of
$\ell^1$ and $\ell^\infty$) is translated into the $\ell^\infty$-norm.

The first reduction is much more complicated, but the core idea is that of
\emph{linear relaxation}.  That is, think of an instance of $\mathrm{SAT}$ as a
constraint satisfaction problem, where each clause is a constraint: is there a
choice of assignments for each clause such that all the variables have the same
value each time?  What Dinur now does is \emph{relax} this problem to get a new
problem $\mathcal{SSAT}_\infty$.  An instance of $\mathcal{SSAT}_\infty$ has the
same input as an instance of $\mathrm{SAT}$, but an allowable assignment for each
clause is a formal $\mathbb{Z}$-linear combination of satisfying assignments.
The value of a variable in the clause is defined as the projection of this vector
to a formal linear combination of variable assignments.  For example, if there
are two variables $x$ and $y$ and two clauses, then
\begin{align*}
  \text{Clause 1: } & \{x=T,y=T\}+\{x=F,y=F\} \\
  \text{Clause 2: } & \{x=T,y=F\}+\{x=F,y=T\}
\end{align*}
is a consistent assignment, because the value of each variable in each clause is
$T+F$.  Despite this extra flexibility, however, a consistent assignment in
which each clause vector has $\ell^1$-norm 1 corresponds to a genuine solution to
the corresponding instance of $\mathrm{SAT}$.  Therefore, the natural norm on
allowable assignments is the maximum over all clauses of the $\ell^1$-norm of the
clause vector, and the set of all consistent assignments forms a lattice in the
vector space spanned by all satisfying assignments of every clause.

The remaining and most complicated ingredient is introducing the gap: how to make
it so that the ``fake'' solutions have very large norm even when the original
$\mathrm{SAT}$ instance is very close to being satisfiable.  This is done using
something analogous to an error-correcting code, via applications of algebra on
finite fields.

\section{The proof of the main theorem for $n=2$}

In this section, we show that it is hard to approximate $L_{\neq 0}(\Sigma)$ or
$L_1(\Sigma)$ when $\Sigma$ is a triangulated surface (possibly with high genus).
We connect this problem to the hardness of approximation results in the last
section by adapting a method given by the second author in \cite[\S5]{GuthSurf}.
Essentially, we show that given a complex $X$ as in the previous section and a
homology class $h \in H_2(X,\ZZ)$, we can construct a triangulated surface
$\Sigma \subset X$ in the homology class $h$, so that the metric on $\Sigma$
closely imitates the metric on $X$.  Then $L_{\neq 0}(\Sigma)$ will be closely
related to $L_{\neq 0}(X, h)$, which is $\mathsf{NP}$-hard to approximate.

By ``closely imitates'', we mean the following.  Given $\epsi>0$, we say that the
$\epsi$-\emph{girth} of a map $p:\Sigma \to X$ is $\delta$ if its image is
$\epsi$-dense in $X$ and the preimages of $2\epsi$-balls have diameter at most
$\delta$.\footnote{This is slightly different from the definition in
  \cite{GuthSurf}.}  It turns out that this condition is sufficient for being
able to extend maps from $\Sigma$ to $X$:
\begin{lem} \label{lem:girth}
  Suppose that $X$ has a cover by $\epsi$-balls with Lebesgue number $\epsi/2$
  and multiplicity $\mu$.  If a map $p:\Sigma \to X$ has $\epsi$-girth $\delta$
  and $f:\Sigma \to S^n$ is a $\delta^{-1}$-Lipschitz map to the unit sphere,
  then there is a map $g:X \to S^n$ such that $f \simeq g \circ p$ and
  $\Lip g \leq C(\mu,n)\epsi^{-1}$.
\end{lem}

Before proving the lemma, we give an example of a map of small girth.  Suppose
that $X$ is a 3-dimensional simplicial complex with equilateral simplices of
sidelength 1.  We build a high genus surface $\Sigma$ and a $1$-Lipschitz map
$p:\Sigma \to X$ as follows.  Let $s>0$ be a small number, and let $\{x_i\}$ be
an $s$-net in $X$.  For each $i$, let $\Sigma_i$ be a 2-sphere of radius $s/10$
and let $p$ project it to $x_i$.  If the distance between $\Sigma_i$ and
$\Sigma_j$ is at most $10s$, then join $\Sigma_i$ and $\Sigma_j$ by a tube of
length at most $10s$ and radius $\sim |X|^{-1/2}s$,\footnote{Here $|X|$ denotes
  the number of simplices of $X$; since we put no restrictions on the
  combinatorics of $X$, the number of facets incident to a vertex may be
  $\sim |X|$.} and let $p$ project the tube to the geodesic between $x_i$ and
$x_j$.  This gives us the surface $\Sigma$ and the map $\sigma$.  We can
triangulate $\Sigma$ with simplices of sidelength $\sim |X|^{-1/2}s$, and the
total number of simplices is $\lesssim s^{-3}|X|^2$.


Clearly $p$ is $\epsi$-dense for any $\epsi>10s$.  Suppose $\epsi>10s$; then any
two points $x_1, x_2$ in $p^{-1}B(x, 2\epsi)$ can be joined by a path in $\Sigma$
of length at most $100\epsi$ as follows.  If the points $x_1, x_2$ are in thin
tubes, we move them to points $x_1', x_2'$ contained in the spheres
$\Sigma_{i_1}, \Sigma_{i_2}$ via paths of length $\lesssim 10s \le \epsi$.  Then
connect $\Sigma_{i_1}$ and $\Sigma_{i_2}$ through a sequence of thin tubes of total
length $\lesssim$ the distance from $\Sigma_{i_1}$ to $\Sigma_{i_2}$, which in turn
is $\le 4\epsi$.  Therefore, the map $p$ has $\epsi$-girth $\delta$ whenever
$\epsi > 10s$ and $\delta > 100\epsi$.


Now we turn to the proof of Lemma \ref{lem:girth}.

\begin{proof}
  Let $\{U_i\}=\{B_\epsi(x_i)\}$ be the assumed cover of $X$.  We first pick a
  partition of unity $\{\phi_i\}$ subordinate to $\{U_i\}$.  This defines a map
  $\phi$ from $X$ to the nerve $N$ of $\{U_i\}$ via barycentric coordinates.
  Choose
  $$\hat\phi_i(x)=\begin{cases}
  1 & 0 \leq d(x_i,x) \leq \epsi/2 \\
  2-2\epsi^{-1}d(x_i,x) & \epsi/2 \leq d(x_i,x) \leq \epsi \\
  0 & \epsi \leq d(x_i,x)
  \end{cases}$$
  and let $\phi_i(x)=\left(\sum_i \hat\phi_i(x)\right)^{-1}\hat\phi_i(x)$.  It is
  not hard to see that with these choices, $\Lip(\phi)=O(\mu^{1/2})\epsi^{-1}$.

  Then we can define a map $\tilde g:N \to S^n$ as follows:
  \begin{itemize}
  \item For each $U_i$, choose a point of $f(p^{-1}U_i)$ and send the
    corresponding vertex of $N$ to that point.
  \item Now for any simplex of $N$, the images of all vertices are contained in a
    1-ball in $S^n$.  Thus we can project to $S^n$ from the affine simplex
    spanned by the vertices; such maps have uniformly bounded Lipschitz constant
    depending only on $n$.
  \end{itemize}
  Defined this way, we can set $g=\tilde g \circ \phi$.  Then for every point
  $x \in \Sigma$, $g \circ p(x)$ is at most 2 units from $f(x)$. 
  Therefore there is a linear homotopy from $g \circ p$ to $f$.
\end{proof}

\begin{rmk}
  This lemma is closely based on Lemma 5.2 in \cite{GuthSurf}, which, however,
  works for a much larger class of target manifolds than just the sphere.  The
  argument above can  probably be generalized, for example by replacing the
  ``simplex flattening'' strategy by barycentric coordinates coming from Karcher
  mean, as in \cite{DVW}.  
\end{rmk}

\begin{proof} [Proof of Theorem \ref{thm:main} for $n=2$.]
  We deal with $L_{\neq 0}$.  The proof for $L_{1}$ is similar.

  Start with $(X,h)$ as in Theorem \ref{hardLXh}, where $X$ is a 3-dimensional
  simplicial complex and $h \in H_2(X, \ZZ)$.  Let $N$ be the number of simplices
  in $X$.  We know that it is $\mathsf{NP}$-hard to distinguish the case that
  $L_{\neq 0}(X,h) \lesssim 1$ from the case that
  $L_{\neq 0} (X,h) \ge N^{c/\log \log N}$.  

  Choose $s = N^{-c/\log \log N}$.  Now given $(X,h)$, we construct a triangulated
  surface $\Sigma_s$ with a map $p:\Sigma_s \rightarrow X$ so that
  $p([\Sigma])=h$ and so that $p$ has $\epsi$-girth $\delta$ for all $\epsi>10s$
  and $\delta>100\epsi$.  The construction is basically the same as the one in
  the paragraph just after the statement of Lemma \ref{lem:girth}.  The only
  problem is that in that description, $[\Sigma] = 0$.  To fix that problem, we
  let $\Sigma_h$ be a cycle in the class $h$.  Then we take $\Sigma_h$ together
  with a dense set of spheres $\Sigma_i$ and join them with thin tubes as above. 

  If $g:X \rightarrow S^n$ with $g^*[S^n](h) \neq 0$, then
  $g \circ p:\Sigma_s \rightarrow S^n$ has nonzero degree and
  $\Lip(g \circ p) \le \Lip(g)\Lip(p) \le \Lip(g)$.  Therefore,
  $L_{\neq 0}(\Sigma_s) \le L_{\neq 0} (X, h)$.  In particular, if
  $L_{\neq 0}(X,h) \lesssim 1$, then $L_{\neq 0}(\Sigma_s) \lesssim 1$ as well.

  On the other hand, suppose that $f:\Sigma_s \rightarrow S^n$ has nonzero
  degree.  If $\Lip(f) \le (1/1000) s^{-1}$, then Lemma 5.1 tells us that $f$ is
  homotopic to a composition $g \circ p$, where $\Lip(g) \lesssim \Lip(f)$.
  Since $f$ is homotopic to $g \circ p$, $g^*[S^n](h) \neq 0$.  Therefore,
  $L_{\neq 0}(\Sigma_s) \gtrsim \min(s^{-1}, L_{\neq 0}(X,h))$.  Recall that we
  chose $s$ so that $s^{-1} = N^{c / \log \log N}$.  So if
  $L_{\neq 0}(X,h) \ge N^{c/\log \log N}$, then
  $L_{\neq 0}(\Sigma_s) \gtrsim N^{c / \log \log N}$ as well.

  Theorem \ref{hardLXh} tells us that it is $\mathsf{NP}$-hard to distinguish the
  case $L_{\neq 0}(X,h) \sim 1$ from the case $L_{\neq 0}(X,h) \ge N^{c/\log \log N}$;
  therefore by the last two paragraphs it is also $\mathsf{NP}$-hard to
  distinguish the case $L_{\neq 0}(\Sigma_s) \sim 1$ from the case
  $L_{\neq 0}(\Sigma_s) \ge N^{c/\log \log N}$.  Therefore, it is $\mathsf{NP}$-hard
  to approximate $L_{\neq 0}(\Sigma)$ to within a factor of $N^{c / \log \log N}$.  

  Recall that $N$ was the number of simplices in $X$, but the number of simplices
  in $\Sigma$ is $N_\Sigma \le s^{-3}N^2$ and so eventually $\leq N^{2.1}$, so the
  same result holds with $N_\Sigma$ in place of $N$.
\end{proof}

\section{Approximating simplicial complexes by spheres}

If $n \ge 3$, the proof of Theorem \ref{thm:main} is similar, except that the
step of constructing $\Sigma$ and the map $p: \Sigma \rightarrow X$ is more
complicated.  If we just wanted to build triangulated $n$-manifolds $\Sigma$ so
that $L_{\neq 0}(\Sigma)$ is hard to approximate, we could use the same
construction, replacing the 2-spheres by $n$-spheres, and replacing the annuli
$S^1 \times [0,1]$ by $S^{n-1} \times [0,1]$.  We will show that if $n \ge 3$,
then $L_{\neq 0} (\Sigma)$ is hard to approximate even for triangulated spheres.
To do that, we have to construct maps of small girth from a triangulated sphere
$\Sigma$ to $X$.  This construction is based on work of Ferry and Okun
\cite{FOkun}.  Because the construction is a little complicated we explain it in
two stages.  First, we construct maps of small girth from $S^n$ to $X$.  This
illustrates the main geometric idea of the construction.  Next, we construct a
triangulated sphere $\Sigma$ with a controlled number of simplices and a
simplicial map of small girth from $\Sigma$ to a refinement of $X$.  For this
simplicial map, we have to check that the number of simplices is not too large
and that the map can be constructed in polynomial time from the data of $(X,h)$. 

\begin{lem} \label{lem:designer}
  Let $X$ be a simply connected triangulated $m$-complex, $n \geq 3$, and
  $\epsi>0$ a constant.  Let $h:S^n \to X$ be any map.  Then there is a piecewise
  Riemannian metric $g$ and a $1$-Lipschitz map $f:(S^n,g) \to X$ homotopic to
  $h$ whose $\epsi$-girth is at most $C(m,n)\epsi$.
\end{lem}
\begin{proof}
  This is a quantitative version of an asymptotic result due to Ferry and Okun
  \cite{FOkun}.  We use a technique adapted from their result which is briefly
  described in \cite[p.~1089]{GuthSurf}, deforming $h$ in several steps.  We
  start by deforming $h$ to a nearby embedding in a thickening of $X$ and giving
  $S^n$ the metric $g_0$ induced by that map.  In the rest of the construction,
  we build up an $(n+1)$-dimensional stratified manifold, equipped with a
  Riemannian metric and a map to $X$, by iteratively attaching $(n+1)$-disks
  along $n$-disks on their boundary.  At every step, the space remains homotopy
  equivalent to $S^n$, and we can think of it as a topologically trivial bordism
  from $(S^n,g_0,h)$ to some $(S^n,g,f)$ which, at the end of the process, will
  have the properties we desire.

  The disks we glue in will take the form of ``balls'', ``sticks'', and
  ``slabs''. That is, we think of them as $D^i \times D^{n+1-i}$, for $i=0$, $1$
  and $2$ respectively, equipped with a map to $X$ which projects to a point,
  immersed curve, or immersed disk, and with the metric induced by the immersion
  in the $D^i$ direction and a round metric of radius $<<\epsi$ in the $D^{n+1-i}$
  direction.

  In detail, the process is as follows:
  \begin{enumerate}
  \item Fix a subdivision $X_\epsi$ of $X$ at scale $\epsi$.
  \item Connect all vertices of $X_\epsi$ in the $\epsi$-neighborhood of $h(S^n)$
    to $h(S^n)$ via geodesic segments.  Call this set of vertices $V_h$ and the
    set of segments $E_h$.
  \item Choose a spanning tree $T$ for the graph $X_\epsi^{(1)}/V_h$.  We attach
    sticks along the edges of $T$ and $E_h$ and balls at the vertices.  We may
    need to homotope $h$ slightly so that it is constant on the $n$-disks where
    we attach the balls.

    At this stage, the map $f_1:(S^n,g_1) \to X$ given by restricting to the end
    of the bordism has $\epsi$-dense image in $X$, but points in $f_1(X)$ which
    are close in $X$ may be far from each other in the induced metric on $f_1(X)$
    (and hence in $g_1$).
  \item Now let $x$ and $y$ be two vertices of $X_\epsi$ which are adjacent via
    an edge not in $T$.  In the next approximation $f_2:S^n \to X$, the portions
    of $f_2(S^n)$ near $x$ and $y$ will be close to each other.  We achieve this
    as follows.

    Let $x'$ and $y'$ be preimages of $x$ and $y$, respectively, in $(S^n,g_1)$.
    There is a path $\gamma$ in $(S^n,g_1)$ between $x'$ and $y'$; since $X$ is
    simply connected, $f_1(\gamma)$ is homotopic to the edge between $x'$ and
    $y'$ via some immersed disk.  We attach a slab which maps to this disk to
    our sphere.  Again we may need to homotope $h$ slightly near the attaching
    $n$-disk.

    Similarly, there may be points in the original sphere whose images are close
    to each other or to vertices of $X_\epsi$ to which they are not attached.  So
    for some $\epsi$-net on the original sphere, consider all pairs of points
    $x$ and $y$ such that $d(f_1(x),f_1(y))<\epsi$ but $d(x,y)>2\epsi$.  For each
    such pair we likewise attach a slab corresponding to a disk homotoping the
    image of a path in $(S^n,g_1)$ to the geodesic path.
  \item Now, $f_2$ still doesn't have small girth, since for a point $p$ deep in
    the interior of a slab, $f_2(p)$ may be close to images of other, distant
    points.  It is enough to create a shortcut from any such $p$ to the ball
    mapping to a nearby vertex $v$ of $X_\epsi$, since then any two such
    points will have a short path between each other through $v$.

    So in each slab $D$, choose an $\epsi$-net $V_D$ and a forest $T_D$
    connecting all the points of $V_D$ to the boundary of $D$.  This $T_D$ can be
    deformed to lie in $X_\epsi^{(1)}$ via a homotopy which moves every point by
    at most $2\epsi$.  Thus we can glue in sticks and slabs spanning each such
    homotopy; after this surgery, each vertex $v \in V_D$ is $2\epsi$-close in
    the induced metric to a point which approximates the nearest vertex of
    $X_\epsi$, and any point in the interior of one of the new slabs is
    $5\epsi$-close to one of these vertices.  In other words, because the new
    slabs are ``narrow'', they do not again have the problem of the slabs we
    added in the previous step.  Thus the resulting map $f$ at the end of the
    bordism has the desired small girth.
    \qedhere
  \end{enumerate}
\end{proof}

In order to perform the reduction in the proof of Theorem \ref{thm:main}, it is
not enough to show that such approximating spheres exist; we must also make sure
that they can be represented by simplicial complexes of polynomial size and
computed in polynomial time.  For this argument, it's important to note that the
complexes constructed in \S\ref{S:hard} are not only simply connected, but
effectively so.  In other words, we can construct a contraction of the
$1$-skeleton whose size is polynomial in the volume of the complex.

Since these complexes are constructed as wedges of spheres with attached mapping
cylinders, we can contract the 1-skeleton by first collapsing the mapping
cylinders onto the wedge of spheres and then puncturing each sphere in some
simplex and contracting away from those points.

Call the relevant complex $X$.  During each of the two stages, loops are expanded
by a factor of $O(|X|)$ and each point travels a distance $O(|X|)$.  In other
words the overall contraction $X^{(1)} \times [0,1] \to X$ can be made simplicial
on a subdivision at scale $\sim 1/|X|^2$.

To complete the argument, we give an algorithmic, simplicial version of the
construction in Lemma \ref{lem:designer}.

\begin{lem} \label{alg:designer}
  Suppose $X$ is a simplicial $m$-complex equipped with an explicit contraction
  $c:X^{(1)} \times [0,1] \to X$, expressed as a simplicial map on a subdivision
  with $V$ facets.  Let $\Sigma_0$ be a triangulation of the $n$-sphere
  and $f_0:\Sigma_0 \to X$ a simplicial map.  Let $\epsi>0$.  Then there is a
  subdivision $X_\epsi$ of $X$ and a triangulation $\Sigma$ of the $n$-sphere and
  a simplicial map $f:\Sigma \to X_\epsi$ such that
  \begin{itemize}
  \item $\Sigma$ has $\poly_n(|X|,|\Sigma_0|,V,\epsi^{-m})$ simplices;
  \item The $\epsi$-girth of $f$ is at most
    $C(m,n)\epsi^{-1} \cdot \operatorname{scale}(X_\epsi)$, when $\Sigma$ is given
    the standard simplexwise linear metric;
  \item $\Sigma$ and $f$ can be computed from $X$, $f_0$, and $c$ in polynomial
    time for fixed $n$.
  \end{itemize}
\end{lem}
\begin{proof}
  We proceed with a version of the construction in Lemma \ref{lem:designer}.  At
  each step, we further subdivide both the domain and the codomain complexes.

  Recall that we can subdivide an $n$-simplex into cubes by adding the midpoint
  of each subsimplex as a vertex.  We can then subdivide those cubes as finely as
  necessary; if we prefer to work with simplicial maps, we can subdivide the
  cubes again to form simplices.  This is somewhat less uniform than the edgewise
  subdivision of \cite{EdGr}, but easier to work with.

  In the beginning, we are given a simplicial map $f_0:\Sigma_0 \to X$.  We start
  by subdividing both the domain and codomain with scale parameter
  $E \sim 1/\epsi$ to form cubical complexes $\Sigma_\epsi$ and $X_\epsi$ so that
  the vertices of $X_\epsi$ form an $\epsi$-net and the edge lengths range
  between about $\epsi$ and $(m+1)\epsi$.

  In order to give ourselves more space to cut and paste, we subdivide
  $\Sigma_\epsi$ with another three cubes to a side to produce $\Sigma_1$.  We
  then precompose the map $f_0$ with a map $\Sigma_1 \to \Sigma_\epsi$ which sends
  the middle cube of each cube of $\Sigma_\epsi$ to the corresponding full cube.

  Now we build a relative spanning tree $T$ for the pair of graphs
  $(X_\epsi^{(1)},f_1(\Sigma_1)^{(1)})$.  This has degree bounded by $C(m)\deg(X)$,
  and so certainly by $C(m)|X|$.  Now fix a cubulation of the $n$-ball with
  enough boundary facets that we can shade $\deg T$ of them and have a path along
  the boundary between any two shaded facets which does not pass through any
  other shaded facets, and which has diameter $C(n)$.\footnote{For example a
    slight subdivision of the simplicial double cone on a sufficiently fine
    triangulation of $D^n$.}  We assemble these, together with $(n+1)$-cubes for
  the sticks, into a cubical thickening of $T$, equipped with maps sending the
  balls to the vertices of $X_\epsi$ and the sticks to the edges.  Finally, for
  each vertex in $f_1(\Sigma_1)$, we arbitrarily choose an $n$-cube in $\Sigma_1$
  mapping to the appropriate vertex and glue it to some boundary facet of the
  corresponding ball.  We let $\Sigma_2$ be the $n$-dimensional ``outer
  boundary'' of the resulting complex, and $f_2:\Sigma_2 \to X_\epsi$ the
  restriction of the associated map.

  Now fix a set of $n$-cells $V_\epsi$ of $\Sigma_2$ which are mapped to vertices
  of $X_\epsi$ and form a net of a certain constant density on $\Sigma_2$; in
  particular, we include a boundary $n$-cell of the balls that we added in as
  well as a single $n$-cell adjacent to each vertex of $\Sigma_\epsi$.

  The next step is to glue in slabs to make sure that the cells of $V_\epsi$
  that map to nearby vertices of $X_\epsi$ are actually nearby.  In order to have
  enough space to do this, we subdivide the domain and codomain again, this time
  with scale parameter
  $$M=10(n+1)|V_\epsi|^2.$$
  Clearly this is polynomial in the various parameters.  We call the resulting
  complexes $\Sigma_M$ and $X_M$.

  Now we find and totally order all pairs $v_i,v_j \in V_\epsi$ so that $f_2(v_i)$
  and $f_2(v_j)$ are equal or adjacent as vertices of $X_\epsi$, but
  $d_{\Sigma_2}(v_i,v_j)>C(n)$, where $C(n)$ is large enough to exclude pairs of
  adjacent vertices in $T$ or $\Sigma_\epsi$.  For each such pair, we build a
  path $p_{ij}$ from $v_i$ to $v_j$ of $n$-cells in $\Sigma_2$ which map to
  $X_\epsi^{(1)}$; a shortest such path can be found in polynomial time after
  building the graph of such $n$-cells.  We can in addition assume that each
  $p_{ij}$ goes ``straight through'' every cell which maps outside $X_\epsi^{(0)}$,
  so that the preimage of each open edge is a disjoint union of open cells each
  of which projects onto that edge by forgetting some coordinates.

  We now build pairwise disjoint paths $p_{ij}^\prime$ of $n$-cubes in $\Sigma_M$,
  where each $p_{ij}^\prime$ is a refinement of $p_{ij}$.  To make sure that they
  don't intersect, we assign each pair $(i,j)$ a distinct distance from the
  middle of cubes of $\Sigma_\epsi$ as well as an ``entry/exit point'' on each
  face, and have the entry/exit stretches and internal stretches map to different
  layers.\footnote{See \cite{KolmBarz} for more details of a similar construction
    in three dimensions, used to fit many maximally separated paths into a cube.}
  The paths $p_{ij}^\prime$ consist of mostly straight chains of cubes with the
  occasional kink.

  For each $i$ and $j$, $f_2(p_{ij}^\prime)$ is a path in $X_\epsi^{(1)}$, and
  (together with the edge from $f_2(v_i)$ to $f_2(v_j)$ if they are not the same)
  it forms a loop $q_{ij}:S^1 \to X_\epsi^{(1)}$.  We would like to glue in a slab
  equipped with a map to a contraction of this loop.  This slab consists of
  several pieces:
  \begin{enumerate}
  \item A thickened annulus which straightens out the kinks in the chain of
    cubes; both sides map to the same path in $X_\epsi^{(1)}$.

    Since the kinks are isolated and come in $\text{const}(n)$ types, this can be
    done with $\text{const}(n) \cdot \text{length}(p_{ij}^\prime)$ vertices,
    perhaps after subdividing both ends a bounded number of times (in which case
    we also must subdivide $\Sigma_M$ and $X_M$.)

    The remainder of the slab will have the product cubical structure of a disk
    with $[0,1]^{n-1}$.
  \item Another thickened annulus which maps to a composition of the path with a
    retraction of $X_\epsi^{(1)}$ to $X^{(1)}$.

    We can perform such a retraction by expanding outward in each direction from
    one central $\epsi$-cube.  This map has Lipschitz constant $\sim \epsi^{-1}$
    and so the inner side of this annulus will consist of
    $$\text{const}(n) \cdot \epsi^{-1} \cdot \text{length}(p_{ij}^\prime)$$
    cubes.
  \item A thickened annulus which maps to the composition of this with the
    provided contraction of $X^{(1)}$.  This consists of
    $$\text{const}(n) \cdot V\epsi^{-1} \cdot \text{length}(p_{ij}^\prime)$$
    cubes.
  \item A thickened disk which maps to the basepoint to which $X^{(1)}$ is
    contracted.  This is a thickening of a cone on the inside of the previous
    annulus.
  \end{enumerate}
  After all the slabs are glued on, we let $\Sigma_3$ be the $n$-dimensional
  ``outer boundary'' of the resulting complex, equipped with a map
  $f_3:\Sigma_3 \to X_M$.

  We now need to algorithmically glue in thin slabs so that the new disks are
  brought close to other parts of the complex which are mapped to nearby points
  of $X_M$.  We first choose an $\epsi$-net of vertices in each disk and the
  shortest path via $n$-cubes from each point in the net to the boundary of the
  disk.  (This means retaining the data of whether a given portion of $\Sigma_3$
  is inherited from $\Sigma_M$.)  Call the set of all the vertices in these nets
  $V_\epsi'$

  Now, for each of these paths, we choose a path in $\Sigma_M$ which maps to a
  nearby path in $X_M$.  Specifically, we make the choice as follows.  Choose a
  sequence of ``signposts'' $x_i$ at intervals of length $\epsi/3m$ along the
  path.  For each signpost $x_i$, choose the nearest vertex $v_i$ of $X_\epsi$.
  Then the path between $x_i$ and $x_{i+1}$ is contained within the star of the
  vertex or edge spanned by $v_i$ and $v_{i+1}$, and we can build a cubical
  approximation to the homotopy which retracts that path linearly to the vertex
  or edge.

  The $v_i$ give us a corresponding path in $X_\epsi$, and we can greedily
  construct a path of $n$-cubes in $\Sigma_M$ which maps to that path, ending at
  the same point on the boundary of the disk.  After subdividing $\Sigma_3$ and
  $X_M$ yet again, this time with scale parameter
  $$M'=10(n+1)|V_\epsi'|,$$
  we use the same technique as before to create non-intersecting sub-paths.  Then
  for each pair of sub-paths, we glue in a slab whose cells map to $X_M$ via the
  piecewise linear homotopy described above.  The number of cubes in this slab is
  $$O(\text{length of the path} \cdot M \cdot (M')^2).$$
  This completes the construction of the space $\Sigma$ and the map $f$.
\end{proof}

\appendix
\section{$\mathsf{NP}$ algorithms for the Lipschitz constant}

In this appendix, we show that $L_{\neq 0}(\Sigma)$ and $L_1(\Sigma)$ can be
approximated up to a constant factor by an $\mathsf{NP}$ algorithm.  In fact,
$\mathsf{NP}$ algorithms can give constant factor approximations to
$\Lip(\alpha)$ in the more general setting of Theorem \ref{lip=comass}.  We need
a little background to state the result precisely.

An \emph{$\mathsf{NP}$ optimization problem} is an optimization problem for
which feasible solutions can be verified in polynomial time.  In other words, it
is a numerical function $f$ from some input set $A$ to $\mathbb{N}$ (or some
other discrete totally ordered set) such that the statement ``$f(x) \leq r$'',
if true, has a polynomial-time certificate.  Then the true optimum can be found
by following all possible certificates in nondeterministic fashion.

\begin{thm} \label{NPopt}
  If $n$ is odd or $m \leq 2n-2$, there is an $\mathsf{NP}$ optimization problem
  which, given an $m$-dimensional simplicial complex $X$ and a class
  $\alpha \in [X,S^n]$, outputs a value within a multiplicative constant $C(m,n)$
  of $\Lip(\alpha)$.
\end{thm}
Informally, we can say that computing the Lipschitz norm up to $C(m,n)$ is an
$\mathsf{NP}$ optimization problem.  In addition, if $\Lip([f]) \geq 1$, then
it can even be computed in polynomial time up to a multiplicative constant using
the algorithm in \S\ref{S:comass}.

The first step of the proof is the following absolute lower bound for
$\Lip(\alpha)$.  

\begin{lem}
  If $\alpha \in [X,S^n]$ is not the zero class, then
  $$\Lip(\alpha) \geq c(m,n)(\#\text{ of simplices of }X)^{-1/n}.$$
\end{lem}
\begin{proof}
  First, a non-nullhomotopic map to $S^n$ always covers $S^n$: otherwise, it is a
  map to the disk and hence nullhomotopic.
	
  Secondly, at the cost of a multiplicative constant, we can restrict to
  simplicial maps from a subdivision of $X$ to a subdivision of $Y$, because of
  the following result:
  \begin{lem}[Quantitative simplicial approximation theorem \cite{CDMW}]
    Let $X$ and $Y$ be simplicial complexes with the standard simplexwise linear
    metric.  Then there is a constant
    $C_{\mathrm{QSAT}}=C_{\mathrm{QSAT}}(\dim X, \dim Y)$ such that every
    $L$-Lipschitz map $f:X \to Y$ can be approximated by a
    $C_{\mathrm{QSAT}}(L+1)$-Lipschitz simplicial map on a subdivision of $X$ at
    scale $\sim 1/L$.
  \end{lem}
  We can choose the subdivision on both sides to be the edgewise subdivision due
  to Edelsbrunner and Grayson \cite{EdGr} with scale parameters $s$ and $t$,
  respectively, indicating the degree of subdivision of each edge.  When we give
  the simplices of the subdivision of $X$ the standard linear metric, the
  resulting metric space differs from an $s$-times dilation of $X$ by a
  multiplicative constant $C(\dim X)$; likewise for $Y$ and $t$.  We call the
  resulting metric spaces $X_s$ and $Y_t$.

  In order to cover $S^n$, a simplicial map must hit every $n$-simplex of $S^n$.
  Thus to admit a surjective simplicial map $X_s \to S^n_t$ without subdividing
  further, $X_s$ must have at least as many $n$-simplices as $S^n_t$.  Now, $X_s$
  has
  $$\leq (\#\text{ of simplices of }X) \cdot C(m) \cdot s^m$$
  $n$-simplices.  In particular, we can choose
  $t=C(m,n)(\#\text{ of simplices of }X)^{1/n}$ such that a surjective simplicial
  map $X_s \to S^n_t$ can only exist when $s \geq 2C_{\mathrm{QSAT}}(m,n)$.  By the
  lemma, this means that any surjective map $X \to S^n_t$ must have Lipschitz
  constant at least $1$, and so any non-nullhomotopic map $X \to S^n$ must have
  Lipschitz constant at least $1/t$.
\end{proof}

Thus, suppose we are given a map $f:X \to S^n$ which we know is not
nullhomotopic.  Then:
\begin{itemize}
\item A homotopic simplicial map $X_s \to S^n_t$, where
  $t=O((\#\text{ of simplices of }X)^{1/n})$ and $s \leq 2C_{\mathrm{QSAT}}(m,n)$,
  is a certificate that $\Lip([f]) \leq s/t$.
\item Moreover, if there is \emph{no} such simplicial map, then
  $\Lip([f]) \geq 1/t$.
\end{itemize}

In our situation, the (polynomial-size) data of a simplicial map $X_s \to S^n_t$
in the right homotopy class is such a certificate, and its validity can be
verified in polynomial time.  In particular, if the homotopy class is specified
using a representative map, then the fact that our map is homotopic to it can be
computed in polynomial time for fixed dimension of $X$ \cite{FiVo}.  This shows
Theorem \ref{NPopt}.

\bibliographystyle{amsplain}
\bibliography{hardapprox}
\end{document}